\documentclass[a4paper]{amsart}

\usepackage{amsmath}
\usepackage{amsthm}
\usepackage{amssymb}
\usepackage{graphicx}
\usepackage{url}

\numberwithin{equation}{section} \swapnumbers

\newcommand{\Lip}{\operatorname{Lip}}
\newcommand{\Fl}{\operatorname{Fl}}

\theoremstyle{plain}
\newtheorem{theorem}{Theorem}[section]

\newtheorem{corollary}[theorem]{Corollary}

\theoremstyle{definition}
\newtheorem{definition}[theorem]{Definition}

\newtheorem{remark}[theorem]{Remark}

\begin{document}

\title[Another approach to some R- and SPDEs]{Another approach to some rough and stochastic partial differential equations}

\author{Josef Teichmann}

\address{ETH Z\"urich, D-MATH, R\"amistrasse 101, 8092 Z\"urich, Switzerland}
\email{jteichma@,math.ethz.ch}
\thanks{The author gratefully acknowledges the support from the FWF-grant Y328}

\begin{abstract}
In this note we introduce a new approach to rough and stochastic partial differential equations (RPDEs and SPDEs): we consider general Banach spaces as state spaces and -- for the sake of simiplicity -- finite dimensional sources of noise, either rough or stochastic. By means of a time-dependent transformation of state space and rough path theory we are able to construct unique solutions of the respective R- and SPDEs. As a consequence of our construction we can apply the pool of results of rough path theory, in particular we obtain strong and weak numerical schemes of high order converging to the solution process.
\end{abstract}
\maketitle

\section{Introduction}
\label{sec:introduction}

There are several, recent and deep approaches to prove existence and uniqueness for stochastic partial differential equations (SPDEs) and rough partial differential equations (RPDEs). We quote pars pro toto \cite{dap/zab92} and \cite{carfriobe09} for the semigroup approach to SPDEs, and for a new deep approach to rough perturbations of fully non-linear PDEs, respectively. What we present in this short note is an alternative approach to solution theory of SPDEs, which can be extended to RPDEs, and allows therefore to treat Banach-space valued stochastic differential equations. We focus on the basic idea behind this alternative method called ``method of the moving frame''. We do neither focus on best estimates nor on all possible conclusions which could be drawn out of those results.

Let $ W $ be a state Banach or Hilbert space. Rough and stochastic partial differential equation
\begin{align}
dY_t = (A Y_t + \alpha(Y_t)) dt + \sum_{i=1}^d f_i(Y_t) dX^i_t = (A Y_t + \alpha(Y_t))dt + f(Y_t) dX_t, \, Y_0 \in W
\end{align}
are usually considered as rough or stochastic perturbation of partial differential equations, which are in the present case of a semi-linear type
\begin{align}
dY_t = (AY_t + \alpha(Y_t)) dt, Y_0 \in W.
\end{align}
Often $ A $ is a partial differential operator, whence only semiflows in one time direction can be constructed, and $ A $ is only densely defined on the state Banach or Hilbert space $ W $. Solution concepts of the rough or stochastic partial differential equation are usually developed from the properties of the deterministic system via variation-of-constants approaches. This approach leads (in contrast to finite dimensional stochastic differential equations) to delicate regularity issues, e.g., regularity of stochastic convolutions, smoothness of vector fields for geometric questions, or the pending boundedness issues for vector fields in RPDEs, when it comes to pathwise definitions of solutions. We propose here one approach to overcome these regularity difficulties which allows for pathwise solutions of SPDEs in the sense of rough paths.

Another motivation comes from mathematical finance: Interest rate evolutions can be described by evolutions of yield curves, which are continuous curves on maturities, i.e., $ Y: \mathbb{R}_{\geq 0} \to \mathbb{R} $. From the point of view of economic theory it makes a lot of sense to consider the derivative of yield curves as properly defined objects, and therefore several norms come by itself, for instance weigthed $ W^{1,1}$-norms on yield curves, or $ C_0 $-norms on the derivatives of yield curves. However, in all the literature on yield curve evolutions in a proper Markovian setting, the usual topology is a Hilbert space topology. This comes with the need to perform stochastic integration to understand the underlying stochastic equations, but has no direct economic interpretation at all. This is a second motivation of this work, namely to find a setting for Banach space valued SPDEs, where yield curve evolutions can be treated with respect to economically meaningful topologies. We believe that rough paths is the tool for it.

We therefore propose here a different point of view on SPDEs, called the ``method of the moving frame'', and we prove several existence and uniqueness results. Some results are beyond what existed in the literature so far -- others are in line. The ``method of the moving frame'' is a time-dependent coordinate transformation applied to the stochastic or rough (partial) differential equations. In the moving frame we obtain a transformed equation, which is a time-depedent SDE or RDE due to the choice of the coordinate transform. Usually this method only works if the generator $ A $ generates a group instead of a semigroup. Here a famous theorem from functional analysis, the Sz\H{o}kefalvi-Nagy theorem (see for instance \cite{Dav:76}),  plays an important role, since it allows us -- at least on Hilbert spaces -- to circumvent elegantly the problem that the semigroup does often not extend to the negative real line.

To be more precise on the method of the moving frame let $ W $ be a possibly infinite dimensional Banach space and $ f_0 $ a possible unbounded, (non-)linear generator of a flow $ \Fl $. The vector fields $ f_i $ and the driving (rough) noise $ X $ are usual, for instance $ f_i $ are $ \Lip(\gamma)$ of some order appropriate for the roughness of the signal $ X $. The method of the moving frame for RPDEs
\begin{align}\label{rpde}
dY_t = f_0(Y_t) dt + \sum_{i=1}^d f_i(Y_t) dX^i_t, \, Y_0 = \xi \in W
\end{align}
is then the construction of a deterministic one-to-one equivalence between the solutions of equation \eqref{rpde} and solutions of
\begin{align}
Y_t = \Fl_t (U_t), \quad dU_t = \sum_{i=1}^d D \Fl_{-t} f_i(\Fl_t(U_t)) dX^i_t, \, U_0 = \xi \in W.
\end{align}
We therefore can simply apply a solution theory of ``simple'' \emph{Banach space valued RDEs} of the type
$$
dU_t = \sum_{i=1}^d g_i(t,U_t) dX^i_t, Y_0 = \xi \in W,
$$
with the usual $ \Lip(\gamma) $-properties on the vector fields in order to construct solutions of RPDEs with \textbf{unbounded drift vector fields $f_0$} through the transformation $ Y_t = \Fl_t(U_t) $.

By this method we can treat solutions for the following partially overlapping problems:
\subsection{RPDEs with group generator on Banach spaces} The construction of unique mild solutions of the rough partial differential equations (RPDE)
\begin{align}
dY_t = A Y_t dt + \sum_{i=1}^d f_i(Y_t) dX^i_t = A Y_t dt + f(Y_t) dX_t,
\end{align}
where $ A $ is the generator of a strongly continuous group acting on the (state) Banach space $ W $. The signal $ X $ is a $p$-geometric rough path for some $ p \geq 1 $. As an example we think of the HJM-equation of interest rate theory acting on spaces of forward rate curves: for instance the space of continuous functions which become constant close to infinity.
\subsection{SPDEs with group generator on Banach spaces} The construction of unique mild solutions of the semi-linear stochastic partial differential equation (SPDE)
\begin{align}\label{spde_bs}
dY_t = ( A Y_t + \alpha(Y_t) )dt + \sum_{i=1}^d f_i(Y_t) \circ dB^i_t,
\end{align}
where $ A $ is the generator of a strongly continuous group acting on the (state) Banach space $ W $. Here $ B $ is a $d$-dimensional Brownian motion. As an example we can think of stochastic heat equations or stochastic Schr\"odinger equations on respective Hilbert or Banach spaces.
\subsection{RPDEs with flow generator on Banach spaces} The construction of unique mild solutions of the non-linear RPDE
\begin{align}
dY_t = f_0(Y_t) dt + \sum_{i=1}^d f_i(Y_t) dX^i_t = V^0(Y_t)dt + f(Y_t)dX_t
\end{align}
where $ f_0 $ is the generator of a strongly continuous flow acting on the (state) Banach space $ W $. The signal $ X $ is a $p$-geometric rough path for some $ p \geq 1 $.
\subsection{SPDE or RPDE with (pseudo-contractive) semigroup generator on Hilbert spaces} The construction of unique mild solutions of the semi-linear stochastic partial differential equation (SPDE)
\begin{align}\label{rpde_hs}
dY_t = ( A Y_t + \alpha(Y_t) )dt + \sum_{i=1}^d f_i(Y_t) \circ dX^i_t,
\end{align}
where $ A $ is the generator of a strongly continuous, pseudocontractive semigroup acting on the separable (state) Hilbert space $ H $. Here $B $ is a $d$-dimensional Brownian motion. This works analogously for RPDEs with pseudo-contractive semigroup generator.

The remainder of the article is structured as follows: in Section \ref{sec:basic_concepts} we introduce the basic concepts of rough path theory, in Section \ref{sec:main} we show the main construction and results. In the final Section \ref{sec:examples_and_conclusions} we present some examples and conclusions, including one from interest rate theory.

\section{Basic concepts from rough paths theory}\label{sec:basic_concepts}

From the theory of rough paths we shall mainly apply the universal limit theorem, which we therefore present in this section. In order to formulate the universal limit theorem we need the notion of $ \Lip(\gamma) $-vector fields, which traces back to Elias Stein. We quote the definition from the St.~Flour lecture notes of Terry Lyons \cite{Lyo:06} and the remarks therein. For the purpose of the definiton we choose a compatible norm $ |.| $ on the tensor products of the Banach space $ V $. In some cases we deal with Hilbert spaces, where all compatible norms are equivalent.
\begin{definition}
Let $V$ and $W$ be Banach spaces. Let $ k \geq 0 $ be an integer and $ k < \gamma \leq k+1 $. Let $F$ be a closed subset of $V$ and $ f: F \to W $ a function and let furthermore $ f^j:F \to L(V^{\otimes j},W) $ for $ j=1,\ldots,k$. We denote $ f^0:=f$. The collection $ (f^0,\ldots,f^k) $ belongs to $ \Lip(\gamma)(F,W) $ if there exists a constant $ M \geq 0 $ such that
$$
\sup_{x \in F} |f^j(x)| \leq M
$$
and if there exists functions $ R_j $ such that, for each $ x,u \in F $ and each $ v \in V^{\otimes j} $, we have
$$
f^j(y)(v) = \sum_{l=0}^{k-j} \frac{1}{l!} f^{j+l}(x)(v \otimes {(y-x)}^{\otimes l}) + R_j(x,y)(v)
$$
and
$$
|R_j(x,y)| \leq M {|x-y|}^{\gamma - j}
$$
holds for $ j=0,\ldots,k$. We usually say that $ f $ is $ \Lip(\gamma)(F,W) $ without mentioning $ f^1,\ldots,f^k $ and we call the smallest constant $ M $ for which all described inequalities hold the $ \Lip(\gamma)$-norm of $f$ and denote it by  $ {||f||}_{\Lip{\gamma}} $.
\end{definition}

For the definition and all details on rough paths we refer to the Terry Lyons' St.~Flour lecture notes \cite{Lyo:06}. 
We fix $p \geq 1 $. We understand a $p$-geometric rough paths $X$ as element of the completion of the set of $1$(-geometric) rough paths with respect to the $p$-variation norm (which is in fact a metric). We denote the set of $p$-geometric rough paths with values in $V$ by $ G \Omega_p(V) $. Notice that $p$-geometric rough paths take values in $G_V^m $. Here $ G^m_V $ denotes the nilpotent Lie group of step $m$ over $V$ embedded into the nilpotent algebra $ \mathbb{A}^m_V $ over $V$, which is often refered as truncated (at level $m$) tensor algebra over $V$.

The $p$-variation norm (distance) $d_p$ is given by
$$
d_p(X,Y)= \max_{1 \leq i \leq [p]} \sup_{\mathcal{D} \subset [0,T]} {\bigl( \sum_{\mathcal{D}} {|| X^i_{t_{l+1},t_l} - Y^i_{t_{l+1},t_l} ||}^{\frac{p}{i}} \bigr)}^{\frac{i}{p}}
$$
defined on the set of finite $[p]$-variation, continuous maps from $ \Delta_T := \{(s,t)| \, 0 \leq s \leq t \leq T \} $ into the algebra $ \mathbb{A}^{[p]}_V $ (the truncated tensor algebra $ T^{[p]}(V) $). The supremum is taken with respect to all sub-divisions $ \mathcal{D} $ of $ [0,T] $. The convergence in the $p$-variation norm can also be expressed by control functions $ \omega $, i.e.~ continuous non-negative functions on $ \Delta_T $ which are super-additive in the sense
$$
\omega(s,t) + \omega(t,u) \leq \omega(s,u), \quad \omega(t,t)=0,
$$
for $ 0 \leq s \leq t \leq u \leq T $. Indeed, for a sequence of positive numbers $ a(n) \to 0 $ as $ n \to \infty $ and a control $ \omega $ we define a notion of convergence for sequences $ X(n) $ to $X$,
$$
|| {X(n)}^i_{(s,t)} || \leq {\omega(s,t)}^{\frac{i}{p}}, \quad || {X}^i_{(s,t)} || \leq {\omega(s,t)}^{\frac{i}{p}}, \quad || {X}^i_{(s,t)} - {X(n)}^i_{(s,t)} || \leq a(n) {\omega(s,t)}^{\frac{i}{p}}.
$$
This convergence means that at least a subsequence converges in the $p$-variation norm, and -- vice versa -- every sequence $X(n)$ converging in the $p$-variation distance to $X$ converges in the previously described controlled sense.

Extensions of $p$-geometric rough $X \in G \Omega_p(V) $ paths by time to $p$-geometric rough paths $ \tilde{X} \in G \Omega_p(\mathbb{R} \times V) $ with the corresponding projection property $ \pi_V (\tilde{X}) = X $ are constructed by (Young-)integrating with respect to the $1$-geometric rough path $ t \mapsto t $.

Let $ f \in \Lip(\gamma-1)(W,L(V,W)) $ for $ \gamma > p \geq 1 $ be fixed and take a rough path $ X \in G \Omega_p(V) $. We consider
\begin{equation}\label{rde}
dY_t = f(Y_t) dX_t, \quad Y_0 = \xi \in W,
\end{equation}
and define that the rough differential equation (RDE) given through \eqref{rde} has a solution if there is a rough path $ Z \in G \Omega_p(V \oplus W) $ such that 
\begin{equation}
dZ_t = h(Z_t) dZ_t, \quad Z_0 = 0
\end{equation}
for $ h \in \Lip(\gamma -1)(V \oplus W,L(V \oplus W)) $ with $ h(x,y)(v,w)=(v,f(y+\xi)(w)) $ and 
$$
\pi_V(Z)=X.
$$ 
The point of this reformulation, which is obviously equivalent in a realm of Young integration (which means $ 1 \leq p < 2 $), is that it also makes sense for rough paths. The universal limit theorem then reads as follows:
\begin{theorem}
Let $ p \geq 1 $ and $ \gamma > p $ be fixed. Let $ f \in \Lip(\gamma)(W,L(V,W)) $ be fixed, then for all $ X \in G \Omega_p(V) $ and all $ \xi \in W $ the RDE
$$
dY_t = f(Y_t) dX_t, \quad Y_0 = \xi
$$
admits a unique solution $ Z = (X,Y) \in G \Omega_p(V \oplus W) $ in the previous sense which depends continuously on $ (X,\xi)$, i.e.~the map
$$
G \Omega_p(V) \times W \to G \Omega_p(V \oplus W)
$$
is continuous in the $p$-variation topology (and continuously extends the factorization of the It\^o map to general rough paths). The rough path $ Y $ is the limit, in the $p$-variation norm, of the Picard-iteration sequence $ Y(n) $ of rough paths defined via
$$
d{Z_t(n)} = h({Z_t(n)})d{Z_t(n)}, \quad Y(n)=\pi_W(Z(n))
$$
for $ n \geq 0 $, with $ Y(0) = 0 $. The convergence is geometric, more precisely, for $ \rho > 0 $ and a control $ \omega $ of the $p$-variation of $ X $ there is $ T_{\rho} \in ]0,T] $ such that for all $ 0 \leq s \leq t \leq T_{\rho} $ and $ i=0,\ldots,[p] $
$$
|| Y(n)^i_{(s,t)} - Y(n)^i_{(s,t)} || \leq 2^i \rho^{-n} \frac{\omega(s,t)^{\frac{i}{p}}}{\beta(\frac{i}{p})!},
$$
where $ T_{\rho} $ depends on the $ \Lip(\gamma)$-norm of $ f $.
\end{theorem}

From the universal limit theorem we can draw one important conclusion, which is likely to have too strong assumptions but will be sufficient for our conceptual framework. An improved version can be found in \cite{LejVic:04}. We consider the time-dependent RDE
\begin{equation}\label{rde_time-dependent}
dY_t = f(t,Y_t)dX_t,
\end{equation}
which is equivalent to the time-dependent RDE
\begin{equation}\label{rde_time-dependent_reformulation}
ds_t = dt, \quad dY_t = f(s_t,Y_t) dX_t, \quad Y_0 = \xi \in W, \quad s_0 = s_0 \in \mathbb{R},
\end{equation}
in the realm of Young integration. The latter equation can be considered as a (time-homogenous) RDE, where we have just stated the universal limit theorem. Therefore the universal limit theorem for time-dependent RDEs reads as follows.
\begin{corollary}
Let $ p \geq 1 $ and $ \gamma > p $ be fixed and $ f(t,y) \in L(V,W) $ a family of endomorphisms for $ t \in [0,T] $ and $ y \in W $. Let's assume $ \tilde{f} \in \Lip(\gamma)([0,T] \times W,L(\mathbb{R} \times V,\mathbb{R} \times W)) $ for the extension
$$
\tilde{f}(s,y)(r,v)= (r,f(s,y)(v)),
$$
and let $ \tilde{X} $ denote the time-extended $p$-geometric rough path in $ G \Omega_p(\mathbb{R} \times V) $ extending $ X $. Then for all $ X \in G \Omega_p(V) $ and all $ \xi \in W $ the RDE
$$
d\tilde{Y}_t = \tilde{f}(\tilde{Y}_t) d\tilde{X}_t, \quad Y_0 = \xi
$$
admits a unique solution $ \tilde{Z} = (\tilde{X},\tilde{Y}) \in G \Omega_p(\mathbb{R} \times V \oplus \mathbb{R} \times W) $ in the previous sense which depends continuously on $ (X,\xi)$, i.e.~the map
$$
G \Omega_p(V) \times W \to G \Omega_p(V \oplus W)
$$
is continuous in the $p$-variation topology (and continuously extends the It\^o map). The corresponding statements on convergence of the Picard iterations also hold true.
\end{corollary}
\begin{remark}
In order to be in the setting of the universal limit theorem one has to extend $ \tilde{f} $ to the whole real line, which is possible by Whitney's extension theorem (see for instance \cite{Lyo:06} and the references therein). Then we can apply the universal limit theorem and read off the respective result on $ [0,T] $.
\end{remark}

\section{Main Existence and Uniqueness Theorems}\label{sec:main}

We apply the previously introduced theory of rough driving signals in order to solve Banach space valued RDEs in a \emph{mild sense}, a notion, which shall be introduced in the sequel. We shall always consider the following setting:
\begin{itemize}
 \item Let $ X $ be a $p$-geometric rough path for some $ p \geq 1 $.
 \item Let $ W $ be a Banach space together with a strongly continuous group $ P $ with generator $ A $.
 \item Let $ f(y) \in L(V,W) $ be a family of endomorphisms for $ y \in W $.
\end{itemize}
We would like to make sense of solutions of the RPDE
\begin{equation}\label{rde_bs-valued}
dY_t = A Y_t dt + f(Y_t) dX_t, \quad Y_0 = \xi \in W.
\end{equation}
In order to do this we consider the concept of a mild solution, which was successfully applied in \cite{dap/zab92} for Hilbert space valued stochastic differential equation, i.e.,
$$
Y_t = P_t \xi + P_t \int_0^t P_{-s} (f(Y_s) dX_s),
$$
where we already have some hope to make sense of the integral. The main idea (and concept) here is to consider instead of the defining equation for mild solutions the time-dependent RDE
\begin{equation}\label{rde_bs-valued_transformed}
dU_t = P_{-t}(f(P_t \, U_t)dX_t), \quad U_0 = \xi,
\end{equation}
where we do not have the unbounded drift term $A$ anymore. If we are able to solve this equation by a $ p $-geometric rough path $ U $, then $ Y_t := P_t U_t $ is a \emph{mild solution of the equation \eqref{rde_bs-valued} by definition}. We speak of a unique mild solution of equation \ref{rde_bs-valued} if the equation \ref{rde_bs-valued_transformed} has a unique solution. In several cases this might be a slightly too strong notion of uniqueness, but we do not have any better in the moment.

We follow our program and state conditions such that equation \eqref{rde_bs-valued_transformed} has a unique solution by applying the corollary of the universal limit theorem of the introductory Section \ref{sec:basic_concepts}.
\begin{theorem}\label{main-thm}
Let $ p \geq 1 $ and $ \gamma > p $ be fixed and $ f(y) \in L(V,W) $ a family of endomorphisms for $ y \in W $. Let's assume $ \tilde{f} \in \Lip(\gamma)([0,T] \times W,L(\mathbb{R} \times V,\mathbb{R} \times W)) $ for the extension
$$
\tilde{f}(s,y)(r,v)= (r,P_{-s}(f(P_{s}y)(v))),
$$
then there exists a unique solution for
$$
dU_t = P_{-t}(f(P_t \, U_t)dX_t), \quad U_0 = \xi
$$
on compact intervals $ [0,T] $ and the universal limit theorem holds true.
\end{theorem}
By this simple theorem on Banach space valued RDEs we are able to solve the previous true RPDE through a mild solution.
\begin{theorem}\label{main}
Let $ p \geq 1 $ and $ \gamma > p $ be fixed and $ f(y) \in L(V,W) $ a family of endomorphisms for $ y \in W $. Let's assume $ \tilde{f} \in \Lip(\gamma)([0,T] \times W,L(\mathbb{R} \times V,\mathbb{R} \times W)) $ for the extension
$$
\tilde{f}(s,y)(r,v)= (r,P_{-s}(f(P_{s}y)(v))),
$$
then there is a unique mild solution for the RPDE
$$
dY_t = AY_t dt + f(Y_t) dX_t, \quad Y_0 = \xi
$$
on compact intervals $ [0,T] $.
\end{theorem}
\begin{proof}
Since a mild solution is by definition a $ P_t $-transformed solution of
$$
dU_t = P_{-t}(f(P_t U_t)dX_t), \quad U_0 = \xi
$$
the result follows immediately by Theorem \ref{main-thm}.
\end{proof}
\begin{remark}
In this setting -- due to the structure of rough paths theory -- the assertions on $ f $ for obtaining mild solutions are best possible, namely the $ \Lip(\gamma) $-property for the time-transformed vector field $ (t,y;r,v) \mapsto (r,P_{-t} (f(P_ty)(v))) $.
\end{remark}
We can extend the previous considerations even to non-linear drift terms, i.e.~we consider the following setting:
\begin{itemize}
 \item Let $ X $ be a $p$-geometric rough path for some $ p \geq 1 $.
 \item Let $ W $ be a Banach space together with a strongly continuous flow $ \Fl $
such that the first variation $ D \Fl_t $ is a well-defined bounded endomorphism for all $ t $. On a dense domain in $W$ we have
$$
\frac{\partial \Fl(t,y)}{\partial t} = f_0(\Fl(t,y))
$$
for some densely defined (non-linear) vector field $ f_0$.
 \item Let $ f(y) \in L(V,W) $ be a family of endomorphisms for $ y \in W $.
\end{itemize}
We would like to make sense of solutions of the RPDE
\begin{equation}\label{rde_bs-valued_non-linear}
dY_t = f_0( Y_t) dt + f(Y_t) dX_t, \quad Y_0 = \xi \in W.
\end{equation}
In order to do this we consider
\begin{equation}\label{rde_bs-valued_non-linear_transformed}
dU_t = D\Fl_{-t}f(\Fl(t,U_t))dX_t, \quad U_0 = \xi,
\end{equation}
where we do not have the unbounded drift term $f_0$ anymore. If we are able to solve this equation by a $ p $-geometric rough path $ U $, then $ Y_t := \Fl_t (U_t)=\Fl(t,U_t) $ is a \emph{mild solution of the equation \eqref{rde_bs-valued} by definition}. A similar theorem as the previous one holds:
\begin{theorem}
Let $ p \geq 1 $ and $ \gamma > p $ be fixed and $ f(y) \in L(V,W) $ a family of endomorphisms for $ y \in W $. Let's assume $ \tilde{f} \in \Lip(\gamma)([0,T] \times W,L(\mathbb{R} \times V,\mathbb{R} \times W)) $ for the extension
$$
\tilde{f}(s,y)(r,v)= (r,D \Fl_{-s}(f(\Fl(s,y))(v))),
$$
then there is a unique mild solution for the RPDE
$$
dY_t = f_0(Y_t) dt + f(Y_t) dX_t, \quad Y_0 = \xi
$$
on compact intervals $ [0,T] $ constructed via $ Y_t = \Fl(t,U_t) $.
\end{theorem}

\section{Examples and Conclusions}\label{sec:examples_and_conclusions}

The examples of this section are chosen with increasing complexity: first we consider a Hilbert space setting, where the semigroup is pseudo-contractive and where we apply the Sz\H{o}kefalvi-Nagy theorem. Second we consider SPDEs with values in Banach spaces. Third we return to the main motivation from mathematical finance, the HJM-equation, and fourth we make some remarks on numerical conclusions.

\subsection{RPDEs with pseudo-contractive semigroups on Hilbert spaces}

Let $ H $ be a Hilbert space and $Q$ a pseudo-contractive, strongly continuous semi-group, i.e.~there is $ \omega \in \mathbb{R} $ such that
$$
||Q_t|| \leq \exp(\omega t)
$$
for $ t \geq 0 $. The Sz\H{o}kefalvi-Nagy theorem tells that there exists a Hilbert space $ W $ containing $ H $ as a closed subspace and a strongly continuous group $ P $ on $ W $ extending $ Q $ in the following sense: for $ h \in H $ we have
$$
Q_t h = \pi \, P_t h,
$$
where $ \pi $ is the orthogonal projection of $ W $ on $ H $. This group extension is constructed by considering first $ {(\exp(-\omega t) P_t)}_{t \geq 0} $ and proving a unitary extension of it, then this unitary extension is multiplied by $ \exp(\omega .) $. The idea for this extension stems from harmonic analysis and can be found in \cite{Dav:76}. Notice that this construction is not a marginal case, which can be rarely applied, but works in many second order differential problems (for instance for many semigroups steming from SDEs). Therefore our main assumption on the \emph{group property} of $P$ is in fact not very strong.

Consider now the RPDE
$$
dY_t = AY_t \, dt + f(Y_t)dX_t
$$
for a $p$-geometric rough path $X$. Here $A$ denotes the generator of $ Q $. By extending $ f $ via $ f\circ \pi $ on $W$ we can formulate the $\Lip(\gamma)$-conditions on $ \widetilde{f} $ as in Theorem \ref{main} in order to obtain mild existence and uniqueness for the previous RPDE.

\subsection{SPDEs with values in Banach spaces}

Notice that the Szek\H{o}falvi-Nagy theorem does not hold on Banach spaces, therefore we have to assume the existence of strongly continuous groups for Banach space valued RPDEs.

It is well-known that beyond certain classes of Banach spaces, namely UMD-spaces, there is no appropriate notion of stochastic integration due to the lack of It\^o isometries or equivalent properties. Therefore also SPDEs with values in Banach spaces are a vital area of research with different competing approaches. We want to add here a new approach which works on \emph{any Banach space $W$} under our usual group (or flow) assumptions. 

Rough path theory is here the via regia to circumvent the lack of proper stochastic integration. We denote by $ X $ a Gaussian $p$-rough path in the sense of \cite{fri/vic:09}, i.e. a Gaussian process where the iterated stochastic integrals up to order $ [p] $ exist almost surely with respective finite $ p $-variation (one component can be time $ t $, too). By Theorem \ref{main} we are able to find unique mild solutions of
$$
dY_t = AY_t \,dt + f(Y_t) dX_t,
$$
under the assumption that $ A $ generates a strongly continuous group on $ W $ and that $ \tilde{f} $ satisfies a $ \Lip(\gamma) $-condition on $ [0,T] \times W $.

So far we cannot really say how this stochastic process $ Y $ is related to what we expect from a solution of an SPDE, the following theorem shows that indeed it coincides with the most natural expectations on a solution:
\begin{theorem}
Let $ X $ be given through a $ d $-dimensional Brownian motion $ B $ extended by time $ B^0_t = t $. Let $ l $ be a continuous linear functional on $ W $, then
$$
Y_t = P_tY_0 + \sum_{i=0}^d \int_0^t l(P_{t-s}f_i(Y_s))) \circ dB^i_s
$$
for $ 0 \leq t \leq T $, where $ \circ dB^0_t = dt $. In particular $ s \mapsto l(P_{-s}f_i(Y_s)) $ is a real-valued semi-martingale.
\end{theorem}
\begin{proof}
For the proof we observe again that $ Y_t = P_t U_t $ for
$$
U_t = U_0 + \int_0^t P_{-s} (f(P_s U_s) dX_s)
$$
in the sense of rough paths. By the universal limit theorem we know that the piecewise linear interpolations of Brownian motion converge on the one hand to the time-extended Brownian rough path, we codify this convergence in $ X(n) \to X $ as $ n \to \infty $. On the other hand -- by stochastic analysis -- we recognize that the integrals along piecewise interpolated Brownian motions converge to the Stratonovich integral of the limit, hence
$$
l(U(n)_t) = l(U_0) + \int_0^tl(P_{-s}f(P_s U(n)_s) dX(n)_s \to l(U_0) + \sum_{i=0}^d \int_0^tl(P_{-s}f_i(P_s U_s) \circ B^i_s
$$
as $ n \to \infty $ holds true.
\end{proof}
\begin{remark}
Similar theorems can be proved for Gaussian rough paths.
\end{remark}

\subsection{HJM equation along rough paths}
Interest rate evolutions can be described by evolutions of yields, which are continuous real-valued curves on times to maturity, i.e., $ Y: \mathbb{R}_{\geq 0} \to \mathbb{R} $ continuous. From the point of view of economic theory it makes a lot of sense to consider the derivative of yield curves as properly defined and therefore several norms come by itself, for instance $ W^{1,1}$-norms on such curves, or $ C_0 $-norms on the derivatives of such curves. The previous construction allows to consider stochastic differential equations with values in such Banach spaces.

More precisely let us consider interest rate evolutions along (Brownian) rough paths with values in a Banach space $W$ of measurable curves on the set of real numbers $\mathbb{R}$, whose restriction on $ \mathbb{R}_{\geq 0} $ corresponds to forward curves (which in turn are closely related to derivatives of yield curves). On this Banach space we consider the shift group $ P_t f(x) = f(x+t) $ and assume it to be strongly continuous thereon. With respect to a martingale measure this equation should read as follows
$$
dr_t = (\frac{d}{dx} r_t + \alpha(r_t)) dt + f_i(r_t) dB^i_t, \quad r_0 \in W,
$$
where $ \alpha(r) =  \sum_{i=1}^d \big( \int_0^. f_i (r) \bigr) f_i (r) $ is the HJM-drift term. If we have the appropriate $ \Lip(\gamma) $ condition for $ \gamma > 2 $ on the time transformed vector fields 
$$ 
P_{-t}\alpha \circ P_t ,\, P_{-t}f_1 \circ P_t,\ldots,\, P_{-t}f_d \circ P_t, 
$$
then we can find a unique mild solution (with respect to the chosen Banach space topology) of the HJM-equation by Theorem \ref{main}.

\subsection{Numerical Conclusions} The theory of rough paths allows to construct approximations of solutions of RPDEs together with corresponding convergence rates. There are different results, which can be quoted from the literatur since we have reduced some classes of RPDEs to time-dependent RDEs: 
\begin{itemize}
\item approximation of a $p$-geometric rough path $ X $ by $1$-rough paths $ X(n) $ in the $p$-variation norm leads to convergence of $ U(n) $ to $ U $ in the $p$-variation norm, with explicit rates known. This can be compared to Wong-Zakai schemes from stochastic analysis.
\item Picard iterations as stated in the universal limit theorem converge with explicitly known rates to the respective limit $ U(n) $.
\item $N$-Euler approximations in the sense of \cite{fri/vic:07} can be established for $U$ and yield high order strong approximation schemes.
\item Cubature formulas in the sense of \cite{lyo/vic04} can be established for $U$ and serve as weak approximation schemes with explicitly known rates.
\end{itemize}

\end{document}